\documentclass[11pt]{amsart}
\usepackage{amsmath, amssymb}
\usepackage{relsize}
\newtheorem{theorem}{Theorem}
\numberwithin{theorem}{section}
\newtheorem{proposition}[theorem]{Proposition}
\newtheorem{lemma}[theorem]{Lemma}
\numberwithin{equation}{section}
\newcommand{\floor}[1]{\lfloor #1\rfloor}

\def\CC{\mathbb C}

\def\NN{\mathbb N}

\def\ZZ{\mathbb Z}

\def\us{\underset}
 
\def\mf{\mathfrak}
\def\eps{\varepsilon}

\begin{document}

\title{Moments of an exponential sum related to the divisor function}
\author[M. Pandey]{Mayank Pandey}
\address {Department of Mathematics, California Institute of Technology, Pasadena, CA 91125}
\email{mpandey@caltech.edu}

\begin{abstract}
We use the circle method to obtain tight bounds on the $L^p$ norm of an exponential sum involving the divisor function for 
$p > 2$.
\end{abstract}
\maketitle

\section{Introduction}
Let $X\ge 1$ be sufficiently large.
For a function $f:\NN\rightarrow\CC$, let 
\[M_f(\alpha) := \sum_{n\le X}f(n)e(n\alpha)\]
where as usual, $e(\alpha) := e^{2\pi i\alpha}$.
Information on the structure of $f(n)$ can be
obtained by studying the size of $L^p$-integrals of $M_f(\alpha)$, and bounds on them are often useful in applications of the circle
method. In particular, often, when bounding the contribution from ``minor arcs" in an application of the circle method, one is led to bounding the $L^\infty$ norm of an exponential sum in the minor arcs times 
$\int_0^1 |G(\alpha)|^p d\alpha$ for some $G(\alpha)$, which is often of the form $M_f(\alpha)$ for various choices of $f$. 

For example, a proof of the minor arc bounds in Vinogradov's theorem
(that all sufficiently large odd integers are the sum of 3 primes) involves bounding $M_\Lambda(\alpha)$ by $O_A(X\log^{-A} X)$
for a particular choice of minor arcs, which along with the fact that by Parseval's identity and the prime number theorem
$\int_0^1 |M_\Lambda(\alpha)|^2d\alpha\ll X\log X$ implies that 
$|M_\Lambda(\alpha)|^3$ is bounded above by $X^2\log^{-A + 1}X$ on average on the minor arcs. 
Here, $\Lambda$ is the von Mangoldt function

Write 
\begin{equation}
I_f(p) := \int_0^1 |M_f(\alpha)|^p d\alpha.
\end{equation}
In the case $p = 1, f = \tau$, it was shown in \cite{GP} that 
\begin{equation}
\sqrt{X}\ll I_{\tau}(1) \ll \sqrt{X}\log X,
\end{equation}
where \[\tau(n) := \sum_{d | n} 1\]
and the methods used to prove the lower bound in that paper should extend to allow one to show that 
\[I_\tau(p)\gg X^{p / 2}\]
for $p < 1$.
For sequences other than $\tau(n)$, similar results have been established in the case $p = 1$.
For example, with $\mu$ the M\"obius function, we have that $X^{1/6}\ll I_\mu(1)\ll X^{1/2}$
where the upper bound follows from Parseval's identity, and the lower bound follows from Theorem 3 in \cite{BR}. 
Estimates for $I_f(1)$ in the case $f$ is an indicator function for the primes have been obtained by Vaughan \cite{Va1} and
Goldston \cite{Go}, and in the case $f$ is the indicator function for integers not divisible by the $r$th power of any
prime by Balog and Ruzsa \cite{BR}.
Later, a result of Keil \cite{Ke} finds with $f$ the indicator function for the $r$-free numbers the exact order of magnitude 
of all moments but $1 + \frac{1}{r}$ in which case the exact order of magnitude is found within a factor of $\log X$.

In this paper, we shall focus on the case $f = \tau$, the divisor function. 
Note that we have that by Parseval's identity
\begin{equation}
\label{eq:L2bound}
I_\tau(2) = \sum_{n\le X} \tau(n)^2 \sim \frac{1}{\pi^2}X(\log X)^3.
\end{equation}
We shall obtain tight estimates on $I_\tau(p)$ for $p > 2$. In particular, we prove the following result. 
\begin{theorem}
We have that for $p > 2$
\begin{equation}
X^{p - 1}(\log X)^p\ll \int_0^1 |M_\tau(\alpha)|^p d\alpha \ll X^{p - 1}(\log X)^p.
\end{equation}
where the implied constants depend only on $p$.
\label{thm:main_thm}
\end{theorem}
\subsection{Notation}
Throughout this paper, all implied constants will be assumed to depend only on $p$ unless otherwise specified. In addition, we write $\lVert\alpha\rVert$ to denote $\inf_{n\in\ZZ} |\alpha - n|$. We write 
$f\asymp g$ to denote that $g\ll f\ll g$ where the two implied constants need not be the same. In addition, any statement with $\eps$ holds for all $\eps >  0$ and implied constants depend on $\eps$ too if it appears.

\section{Preliminaries and setup}
Note that we have that since $\tau(n) = \sum_{d|n}1 = \sum_{uv = n}1$
\begin{align}
\label{eq:hyperbola_trick}
\begin{split}
M_\tau(\alpha) &= \sum_{n\le X}\tau(n)e(n\alpha) = \sum_{uv\le X}e(\alpha uv) 
\\ &= 2\us{u < v}{\sum_{uv\le X}} e(\alpha uv) + \us{u = v}{\sum_{uv\le X}}e(\alpha uv) = 2T(\alpha) + E(\alpha).
\end{split}
\end{align}
Also, let \[v(\beta) := \sum_{n\le X} e(n\beta).\]
We record the following well-known bound on $v(\beta)$ which we will use later.
\begin{lemma}
We have that for $\beta\not\in\ZZ$,  $|v(\beta)|\ll\min(X, \lVert\beta\rVert^{-1}).$
\label{lem:kern_Linfbound}
\end{lemma}
\begin{proof}
Note that we have that if $\beta\in\ZZ$, we are done since then $v(\beta) = \floor{X}$. Otherwise, we have that by summing the geometric 
series, \begin{equation}|v(\beta)| = \bigg\lvert\frac{e((\floor{X} + 1)\beta) - e(\beta)}{e(\beta) - 1}\bigg\rvert 
= \frac{|1 - e(\floor X\beta)|}{|1 - e(\beta)|} = \frac{\sin(\pi\floor X\alpha)}{\pi\alpha}\ll \frac{1}{|\sin 2\pi\alpha|}\ll \frac{1}{\lVert\alpha\rVert}\label{eq:kern_sum}\end{equation}
\end{proof}
In addition, we shall also use the following result on moments of $v(\beta)$.
\begin{lemma}
For $p > 2$, we have that 
\[\int_0^1 |v(\beta)|^p\asymp X^{p - 1}.\]
\label{lem:kern_Lp}
\end{lemma}
\begin{proof}
Note that by the third equality in (\ref{eq:kern_sum}), we have that 
\begin{equation}
\int_0^1 |v(\beta)|^p d\beta \ge\int_{[(2X)^{-1}, (4X)^{-1}]} \sin(\pi\floor X\beta)(\sin{\pi\beta})^{-p} d\beta \gg \int_{[(2X)^{-1}, (4X)^{-1}]} X^{-p}d\beta\gg X^{p - 1}.
\label{eq:kern_lower}
\end{equation}
In addition, note that for positive integers $s$, by considering the underlying Diophantine system, we have that
\begin{align*}
\int_0^1 |v(\beta)|^{2s}d\beta &= |\{1\le x_1,\dots,x_s, y_1,\dots,y_s\le X : x_1 + \dots + x_s = y_1 + \dots + y_s\}|
\\&=\sum_{n\le X}\binom{n - 1}{s - 1}^2\sim 
C_sX^{2s - 1}
\end{align*}
for some $C_s > 0$, so the upper bound, and therefore the desired result, follows from H\"older's inequality.
\end{proof}
We will use the circle method to prove the main result. 
In particular, we shall show that the main contribution to the integral $I_\tau(p)$ comes from $\alpha$ close to 
rationals with small denominator by estimating $I_\tau$ around rationals with small denominator and bounding it 
everywhere else.
To that end, let $$\mf M(q, a) = \{\alpha\in [0, 1]: |q\alpha - a|\le PX^{-1}\}$$
with $P = X^{\nu}$ for $\nu > 0$ sufficiently small, and let 
\[\mf M = \bigcup_{q\le P}\underset{(a, q) = 1}{\bigcup_{a = 1}^q}\mf M(q, a), \mf m = [0, 1]\setminus\mf M.\]
For any measurable $\mf B\subseteq [0, 1)$, let
\[I_f(p;\mf B):= \int_{\mf B} |M_f(\alpha)|d\alpha.\]
We shall prove Theorem \ref{thm:main_thm} by using the fact that $I_\tau(p) = I_\tau(p;\mf M) + I_\tau(p; \mf m)$, showing that
$I_\tau(p; \mf m) = o(X^{p - 1}(\log X)^p)$ and showing that $I_\tau(p;\mf M)\asymp X^{p - 1}(\log X)^p$.

\section{The minor arcs}
Our bound on the minor arcs will depend on the following result, which is nontrivial for $X^{\eps}\ll q\ll X^{1 - \eps}$.
\begin{proposition}
If $|q\alpha - a|\le q^{-1}$ for some $(a, q) = 1$, $q\ge 1$, then 
\begin{equation}
M_\tau(\alpha)\ll X\log(2Xq)(q^{-1} + X^{-1/2} + qX^{-1}).
\end{equation}
\label{prop:minor_arc_ptwise}
\end{proposition}
\begin{proof}
We have that by (\ref{eq:hyperbola_trick}) and the trivial bound $|E(\alpha)|\le X^{1/2}$
\[M_\tau(\alpha) = 2T(\alpha) + O(X^{1/2})\]
so it suffices to show that $T(\alpha)\ll X\log(2Xq)(q^{-1} + X^{-1/2} + qX^{-1})$, since we can 
absorb the $O(X^{1/2})$ into the bound since\[X\log(2Xq)(q^{-1} + X^{-1/2} + qX^{-1})\gg X^{1/2}\log X.\]
To this end, note that by the triangle inequality and Lemma \ref{lem:kern_Linfbound}
\[|T(\alpha)|\le \sum_{u\le X^{-1}}\bigg|\sum_{u < v\le X / u} e(\alpha uv)\bigg|\ll \sum_{u\le X^{1/2}}\min(X/u, \lVert\alpha u\rVert^{-1}).\]
The desired result then follows from Lemma 2.2 in \cite{Va}.
\end{proof}
From this, we obtain the following result.
\begin{lemma}
We have that 
\begin{equation}
I_\tau(p;\mf m)\ll X^{p - 1 - \nu / 2}(\log X)^4.
\end{equation}
\label{lem:minor_arc_full}
\end{lemma}
\begin{proof}
Note that we have that 
\[\int_{\mf m} |M_\tau(\alpha)|^pd\alpha\le \left(\sup_{\alpha\in\mf m} |M_\tau(\alpha)|\right)^{p - 2}\int_{\mf m} |M_\tau(\alpha)|^2d\alpha\ll X(\log X)^3\left(\sup_{\alpha\in\mf m} |M_\tau(\alpha)|\right)^{p - 2}.\]
Suppose that $\alpha\in\mf m$. Then, by Dirichlet's theorem, we have that there exist $a, q$ such that $(a, q) = 1, q\le P^{-1}X, |q\alpha - a|\le PX^{-1}$, so it follows that $q > P$ 
, since otherwise, $\alpha$ would be in $\mf M(q, a)$. Then, by 
Proposition \ref{prop:minor_arc_ptwise}, we have that $|M_\tau(\alpha)|\ll X^{1 - \nu / 2}\log X$, and the desired result follows.
\end{proof}

Now, we proceed to estimate the major arcs. To that end, we first record the following estimate.
\begin{proposition}
For $(a, q) = 1$, $q\ge 1$, we have
\[\sum_{n\le X}\tau(n)e\left(\frac{an}{q}\right) = \frac{X}{q}\left(\log\frac{X}{q^2} + 2\gamma - 1\right) + O((X^{1/2} + q)\log{2q}).\]
\end{proposition}
\begin{proof}
This is shown in the proof of Lemma 2.5 in \cite{PV}. We shall reproduce its proof below.
Note that we have that by (\ref{eq:hyperbola_trick}) \[\sum_{n\le X}\tau(n)e\left(\frac{an}{q}\right) = \sum_{u\le X^{1/2}}\left(\sum_{v\le X/u}2 - \sum_{v\le X^{1/2}}1\right)e(auv/q).\]
For $q\nmid u$, we have that the inner sums are $\ll \lVert au/q\rVert^{-1}$. The contribution from the remaining terms is then 
\[\frac{X}{q}\left(\log\frac{X}{q^2} + 2\gamma - 1\right) + O(X^{1/2})\]
from which the desired result follows.
\end{proof}
Now, it follows then from this and partial summation that for $\alpha\in\mf M(q, a)$, we have
\begin{equation}
\label{eq:major_arc_est}
M_\tau(\alpha) = \frac{1}{q}\left(\log\frac{X}{q^2} + 2\gamma - 1\right)v(\alpha - a/q) + O(X^{1/2 + \nu}\log X).
\end{equation}
Therefore, we have that 
(by using the binomial theorem for $p\in\ZZ^{+}$, and then using H\"older's inequality to bound the remaining error terms)
\[
|M_\tau(\alpha)|^p = q^{-p}(\log X - 2\log q + 2\gamma - 1)^p|v(\alpha - a/q)|^p + O(X^{p - 1 / 2 + \nu}(\log X)^p)
\]
so it follows that
\begin{multline}
\label{major_arc_sums}
\int_{\mf M} |M_\tau(\alpha)|^p d\alpha = \\
\sum_{q\le P}\us{(a, q) = 1}{\sum_{a = 1}^q}\int_{-PX^{-1}}^{PX^{-1}}
 q^{-p}(\log X - 2\log q + 2\gamma - 1)^p|v(\alpha - a/q)|^p d\beta + O(X^{p - 3 / 2 + 4\nu}(\log X)^p)\\
 = \mf S(X, P)\int_{-PX^{-1}}^{PX^{-1}} |v(\beta)|^pd\beta
 + O(X^{p - 3 / 2 + 4\nu}(\log X)^p)
\end{multline}
where 
\[
\mf S(X, P) := \sum_{q\le P}\varphi(q)q^{-p}(\log X - 2\log q + 2\gamma - 1)^p.
\]
It is easy to show that by partial summation, we have
\begin{equation}
\label{Singular series estimate}
\mf S(X, P)\asymp(\log X)^p.
\end{equation}
Also, note that by Lemmas \ref{lem:kern_Linfbound} and \ref{lem:kern_Lp}, we have that 
\[\int_{-PX^{-1}}^{PX^{-1}} |v(\beta)|^pd\beta 
\gg X^{p - 1} - \int_{[PX^{-1}, 1 - PX^{-1}]} X^{-p} d\beta\gg X^{p - 1}.\]
Theorem \ref{thm:main_thm} then folows since this implies that $I_\tau(p; \mathfrak M)\asymp X^{p - 1}(\log X)^p$.

\end{document}